\documentclass[a4paper]{amsart}
\usepackage[utf8]{inputenc}
\usepackage{amsmath}
\usepackage{amsfonts}
\usepackage{amssymb}
\usepackage{lmodern}
\usepackage{thmtools}
\usepackage{amsthm}
\usepackage{cancel}
\usepackage{todonotes}
\usepackage{graphicx}
\usepackage{hyperref}
\usepackage[capitalise]{cleveref}
\usepackage{tikz}
\usepackage{MnSymbol}
\usepackage{breakcites}
\usepackage{multicol}
\usepackage{setspace}
\usepackage[backend=bibtex,
                style=ieee,
                doi=false,
                url=false,
                isbn=false
                ]{biblatex}
\Crefname{enumi}{item}{items}
\creflabelformat{enumi}{#2(#1)#3}
\usetikzlibrary{graphs}

\theoremstyle{plain}
\declaretheorem[name={Theorem}]{theorem}
\declaretheorem[name={Proposition}, sibling=theorem]{proposition}

\declaretheorem[name={Definition}, style=definition]{definition}
\declaretheorem[name={Corollary}, sibling=theorem]{corollary}
\declaretheorem[name={Question},style=remark]{question}

\newcommand{\set}[1]{\{{#1}\}}

\DeclareMathOperator{\Spec}{DgSp}

\DeclareMathOperator{\op}{open}

\newcommand{\A}{\mathcal{A}}
\newcommand{\G}{\mathcal{G}}
\newcommand{\B}{\mathcal{B}}
\newcommand{\C}{\mathcal{C}}
\renewcommand{\S}{\mathcal{S}}

\newcommand{\LR}{\Leftrightarrow}
\newcommand{\ra}{\rightarrow}

\newcommand{\embeds}{\hookrightarrow}
\newcommand{\restrict}[1][]{{\upharpoonright}_{#1}}
\newcommand{\mc}[1]{\mathcal{#1}}
\newcommand{\mrm}[1]{\mathrm{#1}}
\newcommand{\mbf}[1]{\mathbf{#1}}
\newcommand{\mf}[1]{\mathfrak{#1}}

\newcommand{\biemb}{\approx}
\renewcommand{\phi}{\varphi}

\newcommand{\onecomp}{\tikz[inner sep=0pt,baseline=-\the\dimexpr\fontdimen22\textfont2\relax]{
  \node[fill, circle, minimum size=2pt] (0) at (0,0) {};
}}
\newcommand{\twocomp}{\tikz[inner sep=0pt,baseline=-\the\dimexpr\fontdimen22\textfont2\relax]{
  \node[fill, circle, minimum size=2pt] (0) at (0,0) {};
  \node[fill, circle, minimum size=2pt] (1) at (0.5,0) {};
  \draw[->] (0) -- (1);
}}
\newcommand{\twocompsym}{\tikz[inner sep=0pt,baseline=-\the\dimexpr\fontdimen22\textfont2\relax]{
  \node[fill, circle, minimum size=2pt] (0) at (0,0) {};
  \node[fill, circle, minimum size=2pt] (1) at (0.5,0) {};
  \draw[<->] (0) -- (1);
}}
\title{Bi-embeddability spectra and bases of spectra}

\author{Ekaterina Fokina}
\address{Institute of Discrete Mathematics and Geometry, Technische Universit\"at Wien, Wiedner Hauptstraße 8-10/104
Vienna, Austria}
\email{ekaterina.fokina@tuwien.ac.at}

\author{Dino Rossegger}
\address{Institute of Discrete Mathematics and Geometry, Technische Universit\"at Wien, Wiedner Hauptstraße 8-10/104
Vienna, Austria}
\email{dino.rossegger@tuwien.ac.at}

\author{Luca San Mauro}
\address{Institute of Discrete Mathematics and Geometry, Technische Universit\"at Wien, Wiedner Hauptstraße 8-10/104
Vienna, Austria}
\email{luca.san.mauro@tuwien.ac.at}

\thanks{The authors were supported by the Austrian Science Fund FWF through project P 27527.
}



\begin{document}
\begin{abstract}
  We study degree spectra of structures with respect to the bi-em\-bed\-dab\-ility relation. The bi-embeddability spectrum of a structure is the family of Turing degrees of its bi-embeddable copies. To facilitate our study we introduce the notions of bi-embeddable triviality and basis of a spectrum. Using bi-embeddable triviality we show that several known families of degrees are bi-embeddability spectra of structures. We then characterize the bi-embeddability spectra of linear orderings and study bases of bi-embeddability spectra of strongly locally finite graphs.
\end{abstract}
\maketitle
The study of degrees realized by structures is a central topic in computable structure theory initiated by Richter~\cite{richter1981} who was the first to study the degrees of isomorphic copies of structures.
Knight~\cite{knight1986} studied degree spectra of structures, the set of degrees of isomorphic copies of a given structure. Since then the question which sets of degrees can be realized by degree spectra has been widely studied, see for instance~\cite{slaman1998,wehner1998,kalimullin2007,kalimullin2007a,andrews2015,andrews2016}.

In recent years researchers studied degree spectra under equivalence relations other than isomorphism. Fokina, Semukhin, and Turetsky~\cite{fokina2016} gave the following definition.
\begin{definition}
 Given a structure $\A$ and an equivalence relation $\sim$, the \emph{degree spectrum of $\A$ under $\sim$} is
 \[ DgSp_\sim(\A)=\{ deg(\B) : \B \sim\A\}.\]
\end{definition}
Definitions analogous to this were given by Yu Liang~\cite{yu2015} and Montalb\'an~\cite{montalban2015a}.
Under this notion the classical degree spectrum of a structure $\A$ is $DgSp_\cong(\A)$. In~\cite{fokina2016} Fokina, Semukhin, and Turetsky investigated degree spectra under $\Sigma_n$ equivalence, $DgSp_{\equiv_n}(\A)$. Two structures are $\Sigma_n$ equivalent, $\A \equiv_n \B$, if every first order $\Sigma_n$ sentence true of $\A$ is true of $\B$ and vice versa. Andrews and Miller~\cite{andrews2015} studied spectra of theories, the family of degrees of models of a complete theory $T$.
In terms of the above definition the theory spectrum of $T$ is the spectrum of a model $\A$ of$T$ under elementary equivalence, $DgSp_\equiv(\A)$. Recently, Rossegger~\cite{rossegger2018} investigated elementary bi-embeddability spectra of structures, $DgSp_{\approxeq}(\A)$. Two structures are elementary bi-embeddable if either is elementary embeddable in the other. Elementary bi-embeddability lies in between isomorphism and elementary equivalence in the sense that two isomorphic structures are elementary bi-embeddable and two elementary bi-embeddable structures are elementary equivalent but none of these implications reverses. It furthermore lies in between isomorphism and bi-embeddability, the topic of the present study.

Two structures $\A$ and $\B$ are \emph{bi-embeddable}, written $\A\biemb \B$, if either is embeddable in the other.
Bi-embeddability has been studied in both descriptive set theory and computable structure theory. Louveau and Rosendal~\cite{louveau2005complete} proved that the relation of bi-em\-bed\-da\-bi\-li\-ty for countable graphs is complete among analytic equivalence relations  under Borel reducibility; this contrasts to the case of the isomorphism relation which is far from complete on countable graphs. The effective theory behaves quite differently: Fokina, Friedman, Harizanov, McCoy, and Mont\'alban~\cite{fokina2012} proved that the isomorphism relation on several classes of computable structure (e.g., graphs, trees, and linear orderings) is complete among $\Sigma^1_1$ equivalence relations, while Friedman and Fokina~\cite{fokina2012sigma11} observed that the same does not hold  for the bi-embeddability relation on linear orderings (this follows from Montalban's analysis~\cite{montalban2005} of the bi-embeddable type of hyperarithmetic linear orderings, that we will discuss in Section \ref{sec:linorder}). Recently, Bazhenov, Fokina, Rossegger, and San Mauro investigated another computational aspect of the bi-embeddability relation. They studied the complexity of embeddings structures. To facilitate this study, they introduced computable bi-embeddable categoricity~\cite{bazhenov2018computable} and classified the degrees of computable bi-embeddable categoricity for equivalence structures~\cite{bazhenov2017degrees}.

The focus of this paper is the degree spectrum of $\A$ under bi-embeddability, or for short bi-embeddability spectrum of $\A$,
\[ DgSp_\biemb(\A)=\{ deg(\B): \B\biemb \A\}.\]
Obtaining examples of sets of degrees which are, or are not, bi-embeddability spectra of structures is in general difficult, since the bi-embeddability relation does not seem to possess strong combinatorial properties one could use to construct such examples. However, for many of the examples constructed for classical degree spectra, a thorough analysis of their construction shows that their isomorphism spectrum coincides with their bi-embeddability spectrum. Either because the structure is \emph{b.e.\ trivial}, i.e., its isomorphism type and bi-embeddability type coincide, or every bi-embeddable copy computes an isomorphic copy, in which case we say that the structure is a \emph{basis} for its bi-embeddability spectrum.

Given a single structure $\mc A$ we say that $\mc A$ is a $\sim$ basis of $\mc B$ if $\mc A\sim \mc B$ and $DgSp_{\cong}(\mc A)=DgSp_{\sim}(\mc B)$.
Apart from the above observation another motivation to study bases of spectra arises from the comparison of degree spectra under different equivalence relations. Given two equivalence relations $\sim_0,\sim_1$ on structures and a structure $\A$, a common question is if there is a structure $\mc B$ such that $DgSp_{\sim_1}(\mc B)=DgSp_{\sim_0}(\mc A)$.
In general this structure $\mc B$ might look very different than $\mc A$ from a structural point of view. Thus, given $\mc A$ it might be hard to find $\mc B$. Therefore it is useful to restrict $\mc B$ to some specific class of structures. The notion of a basis captures this question nicely for the most restrictive class of structures one could want $\mc B$ to be in, the $\sim_1$ type of $\mc A$. Note that, while our definition of a basis only captures the case where $\sim_0$ is isomorphism, it can be adapted to capture the general case without much effort.

In the present article we study the phenomenon of b.e.\ triviality and bi-em\-bed\-dabil\-ity bases of structures. Thus, if we say that $\mc A$ is a basis of $\mc B$ we mean that $\mc A$ is a bi-em\-bed\-dabil\-ity basis. In \cref{sec:betrivial} we give some examples of b.e.\ trivial structures and use these to obtain examples of well known families of degrees that are realized as bi-embeddability spectra. In \cref{sec:basis} we give a more general definition of a basis where we allow families of structures. This definition is motivated by the notion of basis in topology and linear algebra. In \cref{sec:linorder} we give a complete characterization of the bi-embeddability spectra of linear orderings and in \cref{sec:bespectraslfg} we show that in a subclass of strongly locally finite graphs every structure has a basis consisting of a single structure. We close by stating a number of open questions we consider interesting for future research.
\section{B.e.\ triviality}\label{sec:betrivial}
In this section we show for several families of degrees known to be isomorphism spectra that they are bi-embeddability spectra. For many families in the literature the isomorphism type of the structure realizing it coincides with its bi-embeddability type. We call such structures b.e.\ trivial.
\begin{definition}
 A structure $\A$ is \emph{b.e.\ trivial} if any bi-embeddable copy $\B$ of $\A$ is isomorphic to $\mathcal{A}$.
\end{definition}
A stronger condition that implies b.e.\ triviality is that any endomorphism of a structure is an automorphism. To see that this is strictly stronger consider the infinite complete graph. It is b.e.\ trivial but does have endomorphisms which are not automorphisms.

Apart from being b.e.\ trivial, many families of degrees can be realized by rigid structures, i.e., structures which do not possess non-trivial automorphisms. However, in general there is no connection between the number of automorphisms of a structure and b.e.\ triviality.
\begin{proposition}\label{prop:rigidnotbe}
 There is a b.e.\ trivial structure that is not rigid and there is a rigid structure that is not b.e.\ trivial.
\end{proposition}
\begin{proof}
  For an example of a countable b.e.\ trivial structure that is not rigid consider the complete infinite graph. It is b.e.\ trivial but not rigid. In fact it has continuum many automorphisms as every permutation of its universe is an automorphism.

  On the other hand, for an example of a rigid structure that is not b.e.\ trivial consider a tree in the language of graphs, where the number of successors of a vertex is strictly monotonic in the canonic lexicographical ordering on the tree. This tree is rigid as any automorphism must map a vertex to a vertex with the same number of children. It is however not b.e.\ trivial as it is bi-embeddable with two disjoint copies of itself.
\end{proof}
The complete graph we used in the above proposition as an example of a b.e.\ trivial but not rigid structure is an example of an automorphically trivial structure. Recall that a structure is \emph{automorphically trivial} if there is a finite subset of its universe such that every permutation of its universe that fixes this subset pointwise is an automorphism.
\begin{proposition}
 Automorphically trivial structures are b.e.\ trivial.
\end{proposition}
\begin{proof}
Let $\A$ be automorphically trivial and $\B\biemb \A$. Assume $\mu: \A \ra \B$ and $\nu: \B \ra \A$ are embeddings, and that $S_0$ is a finite substructure of $\A$ such that every permutation of $A$ fixing $S_0$ pointwise is an automorphism. We have that $\B$ is isomorphic to a substructure of $\A$ by $\nu$ and thus every permutation that fixes $\nu(\B)\cap S_0$  is an automorphism of $\nu(\B)$. Let $S_1$ be the pullback of $\nu(\B)\cap S_0$ along $\nu$. Then $S_1$ witnesses that $\B$ is automorphically trivial. We can inductively define $S_{n+1}$ switching the roles of $\A$, $\B$ and $\mu$, $\nu$ when $n$ is odd.
Observe that for all $n$, $S_{n+1}$ is isomorphic to a substructure of $S_n$. Because $S_0$ was finite we will find a fixpoint, i.e., there is an $n$ such that $S_{n+1}\cong S_n$. Let $k$ be the first even number such that $S_{k+1}\cong S_{k}$. Since we constructed $S_{k+1}$ by pulling back $S_k$ along $\nu$ we have that $\nu$ is an isomorphism between $S_{k+1}$ and $S_k$.

We can now build an isomorphism $f:\B\ra\A$. At stage $0$ let $f$ be $\nu\restrict[S_{k+1}]$, the above-mentioned isomorphism between $S_{k+1}$ and $S_k$. At stage $s$, if $f(s)$ is already defined or not in $B$ proceed to the next stage. Otherwise take the least $x\in A$ that is not in the range of $f$ and let $f(s)=x$. Then proceed to the next stage.

Clearly in the limit $f$ will be a bijection between $\B$ and $\A$. To see that it is an isomorphism let $T=dom(f)$ at some stage $s$. We have that $\nu(T)\cap f(T)\supseteq S_k$ and thus there is a permutation $\pi$ of $A$ fixing $S_k$ pointwise such that $\pi (f(T))=\nu(T)$. By automorphic triviality of $\A$ we have that \[f(T)\cong\pi(f(T))=\nu(T)\cong T.\]
Thus, at every stage $s$, $f$ is a partial isomorphism from $\A$ to $\B$ and therefore in the limit an isomorphism.
\end{proof}
Knight~\cite{knight1986} showed that if a structure is automorphically trivial, then its degree spectrum is a singleton, and that otherwise it is upwards closed. By the above proposition also the bi-embeddability spectrum  of automorphically trivial structures is a singleton. Clearly every bi-embeddability spectrum of a structure is the union of the degree spectra of structures in its bi-embeddability type. Thus, Knight's result also holds for bi-embeddability spectra.
\begin{corollary}
 If $\A$ is automorphically trivial, then its bi-embeddability spectrum is a singleton. Otherwise it is upwards closed.
\end{corollary}
We now look at examples of b.e.\ trivial structures that appear in the literature. The following definition appears in~\cite{csima2010}.
\begin{definition}
 Let $X\subseteq\omega$ and $n\in \omega$. The graph $G(\{n\}\oplus X)$ is an $\omega$ chain with an $n+5$ cycle attached to $0$, a $3$ cycle attached to $m$ if $m\in X$ and a $4$ cycle attached to $m$ if $m\not \in X$.
\end{definition}
\begin{proposition}\label{prop:codefamily}
  Let $X\subseteq \omega$, $\mf F$ be a family of sets and $\mc G$ be the disjoint union of the graphs $G(\{ n\} \oplus F)$ for $F\in \mf F$ and $n\in X$. Then $\mc G$ is b.e.\ trivial.
\end{proposition}
\begin{proof}
  It is easy to see that for any set $Y$ and $n\in \omega$, $G(\{n\}\oplus Y)$ is b.e.\ trivial as cycles of length $m$ only embed into cycles of length $m$.

  Now, say $\mc G$ is bi-embeddable with $\A$, say $f: \mc G \ra \mc A$ and $g:\mc A \ra \mc G$. Let $G(\{n\} \oplus F)$ be a component of $\mc G$, then $g(f(G(\{n\} \oplus F)))$ must be in a component containing a substructure isomorphic to $G(\{n\}\oplus F)$. By construction the only component like this is $G(\{n\} \oplus F)$ and as it is b.e.\ trivial we get that $g$ is the inverse of $f$ on $G(\{n\} \oplus F)$.
  We have that for every $n\in X$ and $F\in \mf F$, $\mc G$ contains exactly one component isomorphic to $G(\{n\} \oplus F)$ and no other components. Therefore, $g$ is the inverse of $f$, and thus, $f$ is an isomorphism.
\end{proof}
Graphs of the form required in \cref{prop:codefamily} were used in~\cite{csima2010} to show that the class of non computable degrees and the class of hyperimmune degrees are isomorphism spectra. We now get the same result for bi-embeddability spectra.
\begin{theorem}\label{th:exspectra}
  \begin{enumerate}
    \item For every Turing degree $\mbf{a}$ there is a graph $\mc G$ such that $DgSp_{\biemb}(\mc G)=\{ \mbf{d}: \mbf{d} \geq \mbf{a}\}$.
    \item There is a graph $\mc G$ such that $DgSp_\biemb(\mc G)=\{ \mbf d: \mbf d> \mbf 0\}$.
    \item There is a graph $\mc G$ such that $DgSp_\biemb(\mc G)=\{ \mbf d :\mbf d \text{ is hyperimmune}\}$.
  \end{enumerate}
\end{theorem}
\begin{proof}
  For (1), given a set $X\in \mbf a$ consider the graph using $\{0\}\oplus X$. It is not hard to see that $G(\{0\} \oplus X)$ is b.e.\ trivial and $DgSp_\biemb(G(\{0\} \oplus X))=\{ \mbf d: \mbf d\geq \mbf a\}$.
  Items (2) and (3) follow directly from \cref{prop:codefamily} and the results in~\cite{csima2010}. The proofs given there follow the ideas of Wehner's proof that the noncomputable degrees are the spectrum of a structure~\cite{wehner1998} but with some differences.

  We sketch the proof of (2). Wehner considered the family of finite sets $\mf F=\{ \{n\}\oplus F : \text{ F finite} \land F\neq W_n\}$. He showed that this family is $X$-computably enumerable if and only if $X$ is not computable and coded this family into a structure $\mc H$ such that $\mc H$ is $X$-computable if and only if the family is $X$-c.e. It is unclear how to produce a b.e.\ trivial structure such that $\mf F$ is c.e.\ in every of its bi-embeddable copies. Indeed, one can show that the usual encoding using ``bouquet graphs'' (see~\cite{rossegger2018}) has a computable bi-embeddable copy. However, if we consider the graph $\mc G$ obtained by taking the disjoint unions of the graphs $G(\{n\} \oplus F)$ for $\{n\}\oplus F\in\mf F$ we obtain a b.e.\ trivial structure by \cref{prop:codefamily}. Csima and Kalimullin showed that $\mc G$ is $X$-computable if and only if there is $Y\equiv_T X$ such that for all $e\in \omega$, $Y^{[e]}$ is finite and $Y^{[e]}\neq W_e$.\footnote{Here $Y^{[e]}$ denotes the $e^{\text{th}}$ column of $Y$, i.e., $Y^{[e]}=\{ y: \langle e,y\rangle \in Y\}$.} They then showed that the degrees with this property are exactly the non-computable degrees.
\end{proof}
There are also other spectra known to be bi-embeddability spectra. In~\cite{rossegger2018} Rossegger observed that for all computable successor ordinals $\alpha$ and $\beta$, $\{ \mbf d: \mbf d^{(\alpha)}\geq \mbf 0^{(\beta)}\}$ is the bi-embeddability spectrum of a structure; he constructed b.e.\ trivial structures having such spectra.
It is doubtful whether this result can be extended to include limit ordinals. Soskov~\cite{soskov2013a} gave an example of an isomorphism spectrum of a structure $\A$ such that $DgSp_{\cong}(\A)\subseteq\{ \mbf d: \mbf d \geq \mbf{0}^{(\omega)}\}$ and showed that no structure has $\{ \mbf d: \mbf{d}^{(\omega)}\in DgSp_{\cong}(\A)\}$ as its isomorphism spectrum.
Faizrahmanov, Kach, Kalimullin, and Montalb\'an~\cite{faizrahmanov2018} recently showed that no structure realizes the family $\{ \mbf d: \mbf{d}^{(\omega)} \geq \mbf{a}^{(\omega)}\}$ for $\mbf{a}\geq \mbf{0}^{(\omega)}$ as its isomorphism spectrum.

Andrews and Miller~\cite{andrews2015} showed that the family $\{\mbf d:\mbf d^{(\omega+1)}\geq  \mbf 0^{(\omega\cdot 2+2)}\}$ is not the theory spectrum of a structure. Rossegger's result therefore gives an example of a bi-embeddability spectrum which can not be a theory spectrum.
\section{Basis}\label{sec:basis}
All examples of bi-embeddability spectra seen so far have been realized by examples which are b.e.\ trivial, i.e., their bi-embeddability type and their isomorphism type coincides. B.e.\ triviality is purely model theoretic. Since we are interested in degree spectra, a more general property of a structure $\A$ is when we can find a structure $\B$ in $\A$'s bi-embeddability type such that
\[ DgSp_{\biemb}(\A)=DgSp_{\cong}(\B).\]
In this case we say that $\B$ is a b.e.\ basis for $\A$. We now give a general definition of a basis.
\begin{definition}\label{def:basis}
  Given a structure $\A$ and an equivalence relation $\sim$ we say that a family $\mf B$ of structures is a \emph{$\sim$ basis} for $\A$ if
  \begin{enumerate}
    \item $\forall \B \in \mf B\ \B \sim \A$,
    \item $\forall \B,\C \in \mf B\ DgSp_{\cong}(\B)\not\subseteq DgSp_{\cong}(\C)$,
    \item and $DgSp_\sim(\A)=\bigcup_{\B\in \mf B} DgSp_\cong(\B)$.
  \end{enumerate}
\end{definition}
Recall the notion of Muchnik reducibility; a set of reals $P$ is \emph{Muchnik reducible} to a set of reals $Q$, $P\leq_w Q$, if every real in $Q$ computes a real in $P$. In terms of structures one usually says that $\mc A \leq_w \mc B$ if every structure in the isomorphism type of $\mc B$ computes a structure in the isomorphism type of $\mc A$, which is equivalent to saying that $DgSp_\cong(\B)\subseteq DgSp_\cong(\A)$. Let $\mf{A}$ and $\mf B$ be families of structures. Muchnik reducibility extends naturally to such families.
\[ \mf A \leq_w \mf B :\LR \bigcup_{\mc B \in \mf B} DgSp_\cong(\mc B)\subseteq \bigcup_{\mc A\in \mf A} DgSp_\cong(\mc A)\]
Using this we get the following characterization of a $\sim$ basis.
\begin{proposition}\label{prop:muchnik}
  Let $\mf A$ be the family of structures bi-embeddable with $\A$. The family $\mf B\subseteq \mf A$ is a $\sim$ basis of $\mc A$ if and only if $\mf B$ is a minimum with respect to inclusion such that $\mf B \leq_w \mf A$.
\end{proposition}

All b.e.\ trivial structures exhibited in \cref{sec:betrivial} clearly have a singleton bi-embeddability basis, themselves.
It is unclear whether there exist structures with countable or even finite bi-embeddability basis greater than one. This question can be viewed as the computability theoretic analogue to a conjecture by Thomass\'e in model theory stating that the number of isomorphism types in the bi-embeddability type of a relational countable structure is either $1$, $\aleph_0$, or $2^{\aleph_0}$, see~\cite{laflamme2018} for more on this conjecture.

However, if we consider bases for other equivalence relations, the analogue of this question has been answered positively. Andrews and Miller~\cite{andrews2015} have shown that there is a theory whose degree spectrum is the union of two cones, i.e., there is a complete theory $T$ such that for $\A\models T$, $DgSp_\equiv(\A)=\{\mbf d \geq \mbf a\}\cup \{\mbf d\geq \mbf b\}$ for two incomparable Turing degrees $\mbf a$ and $\mbf b$. Fokina, Semukhin, and Turetsky~\cite{fokina2016} showed that the same holds for $\Sigma_n$ equivalence with $n>2$. Hence, for $\Sigma_n$ equivalence and elementary equivalence there are structures with a basis of size $2$.
\section{Linear orderings}\label{sec:linorder}
Montalb\'an~\cite{montalban2005} showed that all hyperarithmetic linear orderings are bi-emb\-ed\-dab\-le with a computable one, and thus their bi-embeddability spectrum contains all Turing degrees.
The following is a relativization of his theorem~\cite[Theorem 1.2]{montalban2005}.
\begin{theorem}\label{thm:hypXrec}
  Let $X\subseteq \omega$. If a linear ordering is hyperarithmetic in $X$ then it is bi-embeddable with an $X$-computable linear ordering.
\end{theorem}
This theorem implies that every linear ordering has a singleton bi-embeddability basis.

The proof of the original theorem is involved and most of it is not computability theoretic. Its relativization, \cref{thm:hypXrec}, can be obtained by relativizing the computability theoretic part.

As a corollary we obtain a characterization of the bi-embeddability spectra of linear orderings in terms of their Hausdorff rank. Before we state the corollary we introduce the required notions.

\begin{definition}
  Let $\mc L=(L,\leq)$ be a linear ordering. For $x,y\in L$ let $x\sim_0 y$ if $x=y$, for $\alpha$ a countable limit ordinal $x\sim_\alpha y$ if $x\sim_\gamma y$ for some $\gamma<\alpha$ and for $\alpha=\beta+1$ $x\sim_\alpha y$ if the intervals $[[x]_{\sim_\beta},[y]_{\sim_\beta}]$ or $[[x]_{\sim_\beta},[y]_{\sim_\beta}]$ are finite.

  The \emph{Hausdorff rank} of $\mc L$, $r(\mc L)$, is the least countable ordinal $\alpha$ such that $\mc L/{\sim_{\alpha}}$ is finite.
\end{definition}
Hausdorff~\cite{hausdorff1908} showed that a linear ordering is \emph{scattered}, i.e., it does not embed a copy of $\eta$, if and only if it has countable Hausdorff rank. Clearly, if $\mc L$ is not scattered then it is bi-embeddable with $\eta$, and thus has a computable bi-embeddable copy.
In ~\cite{montalban2005} it was shown that a scattered linear ordering is bi-embeddable with a computable one if and only if it has computable Hausdorff rank. Given a set $X\subseteq \omega$ we write $\omega_1^{X}$ for the first non $X$-computable ordinal. We can now state a relativization of this theorem.
\begin{theorem}\label{thm:hrXcom}
  Let $X\subseteq\omega$. A scattered linear ordering $\mc L$ has an $X$-computable bi-embeddable copy if and only if $r(\mc L)<\omega_1^{X}$.
\end{theorem}
In other words, $\mc L$ has an $X$-computable copy if and only if it computes a copy of its Hausdorff rank, i.e., $X\geq_T \mc A\cong r(\mc L)$. This combined with \cref{thm:hypXrec} yields the following characterization of bi-embeddability spectra of linear orderings.
\begin{corollary}\label{cor:basislo}
  Let $\mc L$ be a linear ordering.
  \begin{enumerate}
    \item If $\eta \embeds \mc L$, then $\eta$ is a b.e.\ basis for $\mc L$, i.e., $DgSp_\biemb(\mc L)=DgSp_{\cong}(\eta)=\{ \mbf d: \mbf d\geq \mbf 0\}$,
    \item if $\mc L$ is scattered, then its Hausdorff rank is a b.e.\ basis for $\mc L$, i.e., $DgSp_\biemb(\mc L)=DgSp_{\cong}(r(\mc L))$.
  \end{enumerate}
\end{corollary}
Montalb\'ans result shows that the bi-embeddability spectra of linear orderings of hyperarithmetic Hausdorff rank always have a minimum -- the computable degree. However, this is not the case for all linear orderings.
\begin{proposition}
  Let $\mc L$ be a linear ordering with $\omega_1^{\mrm{CK}} \leq r(L)\leq \omega_1^{\omega_1^{\mrm{CK}}}$. Then $DgSp_\biemb(\mc L)=DgSp_{\cong}(\omega_1^\mrm{CK})$ does not contain a least element.
\end{proposition}
\begin{proof}
  That $DgSp_\biemb(\mc L)=DgSp_{\cong}(\omega_1^{\mrm{CK}})$ follows from \cref{cor:basislo} because $r(\mc L)<\omega_1^{\mrm{CK}}$ and therefore every linear ordering of order type $\omega_1^{\mrm{CK}}$ can compute a linear ordering of order type $r(\mc L)$.

  Goncharov, Knight, Harizanov, and Shore~\cite{goncharov2004} characterized the degrees that compute maximal well ordered initial segments of the Harrison ordering which has order type $\omega_1^{\mrm CK}(1+\eta)$. Let $\mf H$ be the family of these degrees. They showed that $\mf H$ coincides with the family of degrees that compute a $\Pi_1^1$ path through Kleene's $\mc O$. The family of this degrees on the other hand does not contain a minimal element, in particular, it contains a minimal pair of Turing degrees.

  Now, clearly every maximal well ordered initial segment of $\omega_1^{\mrm{CK}}(1+\eta)$ is of order type $\omega_1^{\mrm{CK}}$ and therefore $\mf H\subseteq DgSp_{\cong}(\omega_1^{\mrm CK})$. Note that this already implies that $DgSp_{\cong}(\omega_1^{\mrm{CK}})$ does not contain a least element as $\omega_1^{\mrm{CK}}$ does not have a computable copy and $\mf H$ contains a minimal pair. Nevertheless, we show the other inclusion as well, i.e., $DgSp_{\cong}(\omega_1^{\mrm CK})\subseteq \mf H$. To see this let $\mc A$ be of order type $\omega_1^{\mrm{CK}}$; uniformly partition $\omega$ in disjoint infinite, coinfinite sets $A_i$, and fix a computable ordering $\leq_I$ of order type $1+\eta$ on the natural numbers. Since our sequence of sets $A_i$ is uniform and computable we get computable bijections $f_i:A_i\ra A$. Define an ordering $\leq$ on $\omega$ by
  \[ x\leq y\LR\left( x,y\in A_i\text{ and } f(x)\leq^\A f(y)\right)\text{ or } \left(x\in A_i, y\in A_j, i\neq j \text{ and } i\leq_I j\right).\]
  The ordering defined by $\leq$ has order-type $\omega_1^{\mrm{CK}}(1+\eta)$ and $deg(\leq)=deg(\A)$. Therefore, $DgSp_{\cong}(\omega_1^{\mrm{CK}})\subseteq \mf H$.
\end{proof}

\section{Strongly locally finite graphs}\label{sec:bespectraslfg}
A graph $\mc G$ is \emph{strongly locally finite} if it is the disjoint union of finite graphs, or, equivalently, if all of its connected components are finite. In what follows let $\mf F=\langle F_i\rangle_{i\in \omega}$ be a Friedberg enumeration of the finite connected graphs. We may assume without loss of generality that $\mf F$ is such that we can compute the size $|F_i|$ of every graph $F_i$ uniformly in $i$. Given $x\in \mc G$, let $[x]_{\mc G}$ be the atomic diagram of the component of $x$ and denote by $\ulcorner [x]_\mathcal{G}\urcorner$ the number $i$ such that $\ulcorner [x]_\mathcal{G}\urcorner=F_i$ (if $\mc G$ is clear from the context we omit the subscript).

The \emph{trace} of a graph is the set of indices of finite graphs embeddable into $\mc{G}$, i.e.,
\[ tr(\mc G)=\{ i: F_i\embeds \mc G\}.\]

 The components of $\mc G$ form a preordering $P_\G$ under embeddability, i.e., for $x,y\in \mc G$
\[ [x]\leq_{P_\G} [y]:\LR [x] \embeds [y].\]
We denote by $c(\mc G)$ the \emph{set of components} of $\mc G$, i.e.,
\[c(\mc G)=\{ i : F_i \text{ is isomorphic to a component of } \mc G\}.\]
A component of $\mc G$ is \emph{open} if it belongs to an infinite ascending chain of $P_{\mc G}$, and $\op({\mc G})$ is the subset of $c(\mc G)$ containing all open components of $\mc G$.

We first state some computability theoretic facts about the relations introduced above.
\begin{proposition}\label{lem:slfgfacts}
  Given a strongly locally finite graph $\mc G$ and $x,y\in G$,
  \begin{enumerate}
    \item\label{it:slfg1}$y\in [x]_{\mc G}$, $tr(\mc G)$ are $\Sigma^{\mc G}_1$,
    \item\label{it:slfg2}and $[x]\embeds [y]$, $|[x]|\leq |[y]|$, $[x]\cong [y]$, $c(\mc G)$ are $\Sigma^{\mc G}_2$.
  \end{enumerate}
\end{proposition}
\begin{proof}
  Ad \cref{it:slfg1}: For $x\in \mc G$, $[x]_{\mc G}$ is definable by the following $\Sigma_1$ formula.
  \[ y\in [x]_{\mc G} \LR \bigvee_{n\in \omega} \exists u_1,\dots u_n \bigwedge_{1\leq i\neq j\leq n} u_i E u_j\]
  Given $x\in \mc G$ with $|[x]| = n$, let $D([x])(x_1,\dots, x_n)$ be the formula obtained by replacing every constant in the atomic diagram of $[x]$ by a variable. Note that given $n$ we can computably define $D([x])(x_1,\dots, x_n)$ and that for $F_i$ we can obtain $n$ computably. Thus the trace of $\mc {G}$ is definable by the following $\Sigma_{1}$ formula.
  \[ x\in tr(\mc G) \LR \exists x_1,\dots , x_n\  D(F_x)(x_1\dots x_n)\]
  Ad \cref{it:slfg2}: In general, given $x\in \mc G$ the size of its component $[x]$ is $\Sigma_2$ as
  \[ |[x]|=n \LR \exists x_1,\dots x_{n} \bigvee_{1\leq i \leq n} x_i\in [x] \land \forall y (\bigvee_{1\leq i \leq n} x_i\neq y \ra \bigvee_{1\leq i \leq n}\neg x_i E y).\]
  Then
  \[ [x]\embeds [y]\LR \bigvee_{n\in\omega} |[x]|=n \land \exists y_1,\dots y_n D_\exists([x])(y_1,\dots,y_n),\]
  which is $\Sigma_2$.
  Thus also $|[x]|\leq |[y]|$ is $\Sigma_2$ and $[x]\cong [y]\LR [x]\embeds [y] \land [y]\embeds [x]$ as $[x]$ and $[y]$ are finite; hence, it is also $\Sigma_2$. By definition, $x\in c(\mc G)$ if and only if $\exists y\in \mc G\ F_x \cong [y]$ which by the above arguments is $\Sigma_2$.
\end{proof}
\begin{definition}
  A graph $\mathcal{G}$ is \emph{open-ended} if every component of $\mc G$ is open.
\end{definition}
We say that a graph $S^{\mc G}$ is the \emph{skeleton} of $\mc G$ if $S^{\mc G}\cong \bigcup_{i\in tr(\mc G)} F_i$. It is not hard to see that two bi-embeddable graphs $\mc A$, $\mc B$ have the same trace, and thus the same skeleton. For open-ended strongly locally finite graphs the skeletons form a basis.
\begin{theorem}\label{th:openendedbebasis}
  Let $\mc G$ be an open-ended strongly locally finite graph, then $S^{\mc G}$ is a b.e.\ basis of $\mc G$.
\end{theorem}
\begin{proof}
  We first show that $\mc G$ and $S^{\mc G}$ are bi-embeddable given that $\mc G$ is open-ended. Given enumerations of the components of $\mc G$ and $S^{\mc G}$, say we have defined an embedding $\mu$ on the first $s$ components of the enumeration of $\mc G$ and want to define it for the component with index $s+1$ in the enumeration. As $\mc G$ is open-ended, so is $S^{\mc G}$; thus, there is a component which is disjoint from the range of $\mu$ and in which the component with index $s+1$ embeds; define $\mu$ accordingly. It is then not hard to see that in the limit $\mu$ is an embedding of $\mc G$ in $S^{\mc G}$. By the same argument we can embed $S^{\mc G}$ in $\mc G$.

  By \cref{prop:muchnik}, it remains to show that $S^{\mc G}$ is minimal with respect to Muchnik reducibility, i.e., that every $\mc A\biemb \mc G$ computes a copy of $S^{\mc G}$. By \cref{lem:slfgfacts}, $tr(\mc A)$ is $\Sigma_1^{\mc A}$.
  Let $W_e^{\mc A}=tr(\mc A)$ and $W_{e,s}^{\mc A}$ the approximation to $W_e^{\mc A}$ at stage $s$. We construct the copy of $S_{\mc G}$ in stages. At every stage $s$ check if any $i<s$ enters $W_{e,s}^{\mc A}$ and if so build a component isomorphic to $F_i$ using elements bigger than $s$ not yet used during the construction.\footnote{We assume without loss of generality that no $i$ may enter $W_{e,s}$ at a stage $s$ smaller than $i$}
  As the construction is $\mc A$-computable and $tr(\mc A)=tr(\mc G)$, the constructed structure is an $\mc A$-computable copy of $S^{\mc G}$.
\end{proof}

Notice that we can reformulate \cref{th:openendedbebasis} as follows. For any open-ended graph $\G$, we have that
\begin{equation*}
\Spec_{\approx}(\G)=\set{\deg(Y) :\ tr(\G) \mbox{ is c.e.\ in $Y$}}.
\end{equation*}

This is close to the definition of \emph{enumeration degree} of a structure $\S$ as given by Montalb\'an~\cite{montalban2017} in the spirit of Knight~\cite{knight1998}.

\begin{definition}\label{def:enumdegree}
A structure $\S$ has \emph{enumeration degree} $X\subseteq \omega$ if the following holds
\[
\Spec_{\cong}(\S)=\set{\deg(Y) :\ X \mbox{ is c.e.\ in $Y$}}.
\]
\end{definition}
Related to this is the notion of the jump degree of a structure.
\begin{definition}\label{def:jumpdegree}
  A structure $\S$ has jump degree $X\subseteq \omega$ if $\deg(X)$ is the least degree in \[\Spec_{\cong}'(\S)=\{\mbf{d'}:\mbf{d}\in \Spec(\A)\}\]
  The set $\Spec_{\cong}'(\S)$ is often called the \emph{jump spectrum} of $\S$.
\end{definition}
Coles, Downey, and Slaman~\cite{coles2000} showed that for any set $X\subseteq \omega$ the set $\{ \mbf{d}': X \text{ is c.e.\ in }\mbf{d}\}$ has a minimum. It follows from this that a structure has jump degree if it has enumeration degree.

Examples of classes of structures always having an enumeration degree are
algebraic fields (see Frolov, Kalimullin, and Miller~\cite{frolov2009}) and connected, finite-valence, pointed graphs  (see Steiner~\cite{steiner2013}). Bi-embeddability spectra of open-ended graphs are therefore similar to isomorphism spectra of structures in these classes. 

\begin{theorem}\label{thm:openendedthm}
\begin{enumerate}
\item\label{it:openendednotcone} For every $X\subseteq \omega$ there is an open ended graph $\G$ such that $tr(\G)\equiv_e X$.
\item For all open-ended $\G$, $\Spec_{\approx}'(\G)=\set{\mathbf{d}' :\ \mathbf{d}\in \Spec_{\approx}(\G)}$ is a cone of degrees.
\end{enumerate}
\end{theorem}
\begin{proof}
The idea of the proof is similar to that given in~\cite[Corollary 1]{frolov2009}.

\begin{enumerate}
\item Let $X\subseteq \omega$ and define $\G$ to be the graph conisting of a cycle of length $n$ for every $n\in X$.

We have $tr(\G)\equiv_e X$. Indeed, to enumerate $X$ from an enumeration of  $tr(\G)$, enumerate $tr(\G)$ and for every $x\in tr(\G)$ check in a c.e.\ way if $F_x$ is a cycle. If so enumerate the length of the cycle. Clearly this is an enumeration of $X$. On the other hand given an element $x\in X$, consider the trace of the cycle of length $x$ and enumerate it. By \cref{lem:slfgfacts} this is c.e. Thus, given an enumeration of $X$ we can produce an enumeration of $tr(\G)$.
\item Given an open-ended $\G$, by the above mentioned result by Coles, Downey, and Slaman~\cite{coles2000}, the set of jumps of degrees enumerating $tr(\G)$ has a minimum. By \cref{th:openendedbebasis} this is $\Spec_{\approx}'(\G)$.
\end{enumerate}
\end{proof}
\begin{corollary}
  There is a open-ended graph such such that $\Spec_{\approx}(\G)$ does not have a least element.
\end{corollary}
\begin{proof}
  Take $X\subseteq \omega$ to be non-total. It follows that the set of Turing degrees enumerating $X$ does not have a least element. Then by (1) of \cref{thm:openendedthm} and the observation after \cref{th:openendedbebasis} we get that there is $\G$ such that
  \[ DgSp_{\approx}(\G)=\{ deg(Y): tr(\G) \text{ is c.e.\ in Y}\}=\{ deg(Y): X \text{ is c.e.\ in Y }\}.\]
  Therefore $\Spec_{\approx}(\G)$ does not have a least element.
\end{proof}
It is immediate from the construction in \cref{it:openendednotcone} that $\mc G\equiv_T tr(\mc G)\equiv_T D$. Thus $\mc G$ has enumeration degree with respect to its bi-embeddability type and, as it is b.e.\ trivial, also with respect to its isomorphism type.
\section{Open question}
We close by stating a few open question which make for interesting further research.
\begin{question}
  Is there a bi-embeddability spectrum that is not an isomorphism spectrum and vice versa?
\end{question}
\begin{question}
  Is there a structure having finite bi-embeddability basis greater than $1$?
\end{question}
\begin{question}
  Is there a structure having countable bi-embeddability basis?
\end{question}
\begin{question}
  Is there a structure $\A$ and incomparable Turing degrees $\mbf a$ and $\mbf b$ such that $DgSp_{\biemb}(\A)=\{\mbf d\geq\mbf a\}\cup\{\mbf d\geq\mbf b\}$?
\end{question}
\printbibliography
\end{document}